\documentclass[12pt]{article}
\usepackage[cp1251]{inputenc}
\usepackage{amsmath,amsthm,mathrsfs}
\usepackage[english]{babel}

\usepackage{amsfonts,amstext,amssymb,verbatim,epsfig}
\usepackage{pgf,tikz}
\usepackage{float}
\usetikzlibrary{arrows}
\usepackage{hyperref}
\usepackage{cite}
\usepackage{epstopdf}
\usepackage{color}
\usepackage{transparent}
\usepackage{subfigure}

\usepackage{xcolor}
\usepackage[normalem]{ulem}

\hypersetup{colorlinks=true,pdfborder=000,  citecolor  = blue, citebordercolor= {magenta}}
\hypersetup{linkcolor=black}
\makeatletter
\let\reftagform@=\tagform@
\def\tagform@#1{\maketag@@@{(\ignorespaces\textcolor{purple}{#1}\unskip\@@italiccorr)}}
\renewcommand{\eqref}[1]{\textup{\reftagform@{\ref{#1}}}}
\makeatother

\makeatletter

\DeclareUrlCommand\ULurl@@{%
  \def\UrlLeft{\uline\bgroup}%
  \def\UrlRight{\egroup}}
\def\ULurl@#1{\hyper@linkurl{\ULurl@@{#1}}{#1}}
\DeclareRobustCommand*\ULurl{\hyper@normalise\ULurl@}
\makeatother

\def\lessim{\ \lower4pt\hbox{$
		\buildrel{\displaystyle <}\over\sim$}\ }
\def\gessim{\ \lower4pt\hbox{$\buildrel{\displaystyle >}
		\over\sim$}\ }

\def\la{\langle}
\def\ra{\rangle}

\newcommand{\e}{\mathbb{E}}
\newcommand{\p}{\mathbb{P}}

\newcommand{\bg}{\boldsymbol{\gamma}}

\newcommand{\bb}{\boldsymbol{\beta}}
\newcommand{\bm}{\boldsymbol{m}}

\newtheorem{lemma}{\bf Lemma}

\newtheorem{theorem}{\bf Theorem}
\newtheorem{corollary}{\bf Corollary}
\newtheorem{remark}{\bf Remark}

\newenvironment{Proof of lemma}{\noindent{\bf Proof of Lemma}}{\hfill$\Box$\newline}
\newenvironment{Proof of theorem}{\noindent{\bf Proof of Theorem}}{\hfill{\footnotesize${\square}$}\newline}
\newenvironment{Proof of theorems}{\noindent{\bf Proof of Theorems}}{\hfill$\Box$\newline}
\newenvironment{Proof of proposition}{\noindent{\bf Proof of Proposition}}{\hfill$\Box$\newline}
\newenvironment{Proof of propositions}{\noindent{\bf Proof of Propositions}}{\hfill$\Box$\newline}
\newenvironment{Proof of exercise}{\noindent{\it Proof of Exercise:}}{\hfill$\Box$}
\begin{document}

\title{A duality principle in spin glasses}

\author{Antonio Auffinger \thanks{Department of Mathematics. Email: tuca@northwestern.edu} 
        \\ \small{Northwestern University} 
	\and Wei-Kuo Chen  \thanks{School of Mathematics. Email: wkchen@umn.edu}
	\\ \small{University of Minnesota}
	}\maketitle


%

\begin{abstract}
We  prove a duality principle that connects the thermodynamic limits of the free energies of the Hamiltonians and their squared interactions. Under the main assumption that the limiting free energy is concave in the squared temperature parameter, we show that this relation is valid in a large class of disordered systems. In particular, when applied to mean field spin glasses, this duality provides an interpretation of the Parisi formula as an inverted variational principle, establishing a prediction of Guerra \cite{G15}.
\end{abstract}

\maketitle


	\section{Introduction} A fundamental goal in spin glass theory is to study the thermodynamic limit of random Hamiltonian systems that simultaneously exhibit ferromagnetic and anti-ferromagnetic properties. Roughly speaking, one seeks to obtain information on macroscopic phenomena such as phase transitions through the analysis of stochastic microscopic interactions. Applications of spin glass theory include many problems in the fields of physics, biology, neurology and computer science, see for instance \cite{MPV}. More importantly, it also has led to a vast and challenging collection of beautiful mathematical questions \cite{T03}.

	For each $N\geq 1$, let $\Sigma_N$ be a configuration space and $\nu_N$ be a random  measure on $\Sigma_N$. Let $H_N$ be a random Hamiltonian indexed by $\Sigma_N.$ In this paper, we study the thermodynamic limit of the free energy associated to the (inverse) temperature $\beta\in\mathbb{R},$ 
	$$F(\beta):=\lim_{N\rightarrow\infty} \frac{1}{N} \log Z_{N}(\beta),$$
	where $$Z_N(\beta):=\int_{\Sigma_N}\exp \big( \beta H_N(\sigma)\big)\nu_N(d\sigma)$$ is called the partition function. 
	There are two classical viewpoints that are often employed to guess or compute an expression for this limit. The first one is physicists' entropic principle in statistical mechanics. It suggests that $F(\beta)$ has a variational representation that involves a maximization between the entropy of a generic thermodynamical state and its internal energy. Another approach is via large deviation theory, from which  Laplace-Varadhan's lemma writes $F(\beta)$ as a maximization problem between an energy functional and a rate function. The reader interested in these classical approaches can see \cite{DZ, Ellis} and the references therein.
	
	 In contrast to the classical methods, in the ground breaking treatment of the famous Sherrington-Kirkpatrick (SK) model, Giorgio Parisi \cite{Parisi} predicted that the thermodynamic limit of the free energy can be computed through a sharply different variational problem. Introducing the notion of replica symmetry breaking, he expressed $F(\beta)$ as a minimization problem of a functional $\mathcal{P}_{\beta}$ over the space of all probability measures on the interval $[0,1]$. The minimizer of this variational problem is known as the functional order parameter and plays a fundamental role in describing the system. Following Guerra's beautiful discovery of the replica symmetry breaking bound, Parisi's formula was firstly verified in the celebrated work of Talagrand \cite{Talagrand} and then extended to all mixed $p$-spin models by Panchenko \cite{Panchenko}. Further generalizations have also been pushed forward in the generalized random energy model \cite{BK,BK2}, the spherical mixed $p$-spin model \cite{C13}, the Ghatak-Sherrington model \cite{P05}, the multi-species SK model \cite{P15}, the mixed $p$-spin model with vector spins \cite{P152} and the Potts spin glass model \cite{P153}.
	 
	 Despite all remarkable progress during the past 35 years, some aspects of the Parisi solution remain to be understood. In particular, it is unclear how to explain the Parisi solution directly from the aforementioned classical methods. The aim of this paper is to provide a rigorous framework that connects the Parisi formula and the classical approaches. It is motivated by an observation of Francesco Guerra \cite{G15}, who suggested that the thermodynamic limit of the free energy in the mixed $p$-spin model is concave in the squared temperature. From such concavity, he conjectured a Legendre duality between the Parisi formula and the Legendre transform $\Gamma$ of the scaled Parisi functional $\mathcal{P}_{\sqrt{\beta}}$, where the temperature and the functional order parameter are conjugate variables. This duality was rigorously established for the mixed $p$-spin models in the recent work of Auffinger and Chen \cite{AC15}, but no interpretation of the functional $\Gamma$ was given.

	In this paper, we show that Guerra's insight not only holds for the mixed $p$-spin glass model, but also extends to many  disordered models. We establish that if the limiting free energy in the squared temperature is concave, then the system will exhibit a general Legendre duality principle between the limiting free energies corresponding to the original and squared Hamiltonians. Analogously to the classical entropic principle where temperature and energy parameters are Legendre conjugates, this duality also connects these two parameters. As an immediate consequence, we recover the Legendre structure obtained in \cite{AC15}. Foremost, we give a direct interpretation of the functional $\Gamma$, as the limiting free energy of the squared Hamiltonian. While the Parisi formula relates temperature parameter with the Parisi measure, our new representation  conjugates energy and temperature parameters, as in the classical approaches. From these, we also derive a new maximum variational representation for the limiting free energy, which shares some features similar to those obtained through large deviation principles in many classical examples, see \cite{Ellis,RT}.
	 
	 In addition, we apply our Legendre duality to the random energy model, from which we obtain the corresponding Parisi formula previously derived by Guerra \cite{G13} and Bolthausen-Kistler \cite{BK}. Our result extends to the spherical mixed $p$-spin model as well, where it explains the nature of the Crisanti-Sommers representation of the limiting free energy.

	{\noindent \bf Acknowledgements.} Both authors thank Francesco Guerra for bringing their attention to the papers \cite{G13, G15} and for several useful suggestions. We also thank Dmitry Panchenko for valuable comments on the presentation of the paper. The research of A. A. is partly supported by NSF grant DMS-1597864 and NSF CAREER Grant DMS-1653552. The research of W.-K. C. is partly supported by NSF grant DMS-1642207 and Hong Kong Research Grants Council GRF-14302515. 
	
	 \section{Legendre duality}\label{sec1}
	 \subsection{Main results}\label{sec1.1}
For each $N\geq 1,$ let $\Sigma_N$ be a configuration space and $\nu_N$ be a random measure on $\Sigma_N$ with $\nu_N(\Sigma_N)<\infty$. Suppose that $(H_{N,p})_{p\geq 1}$ is a sequence of independent random Hamiltonians indexed by $\Sigma_{N}$ and is independent of $\nu_N$. Denote the normed vector space $\ell^{2}(\mathbb N) $ as $$\mathcal{T}=\bigg \{\bb=(\beta_p)_{p\geq 1}: \sum_{p\geq 1}\beta_p^2<\infty \bigg \}.$$  The sequence $\bb$ here is called the (inverse) temperature parameter. 
	 Throughout Section \ref{sec1}, we assume that the following conditions hold:
	 \begin{itemize}
	 	\item[(A0)] Suppose that for each $p\geq 1$, the Hamiltonian $H_{N,p}$ is symmetric, that is, the families of random variables $(H_{N,p}(\sigma))_{\sigma\in\Sigma_N}$ and $(-H_{N,p}(\sigma))_{\sigma\in\Sigma_N}$ share the same joint distribution.
	 	
	 	\item[(A1)]  There exist  constants $C_1,C_2>0$ such that for any $N\geq 1$ and $p\geq 1$, 
	 	$$\e \sup_{\sigma\in\Sigma_N} H_{N,p}(\sigma)\leq \frac{C_1N}{2^{p/2}},
	 	$$
	 	and
	 	\begin{align*}
	 	\p\Bigl(\Bigl|\sup_{\sigma\in\Sigma_N}H_{N,p}(\sigma)-\e\sup_{\sigma\in\Sigma_N}H_{N,p}(\sigma)\Bigr|\geq t\Bigr)\leq C_2\exp\Bigl(-\frac{2^{p}t^2}{C_2N}\Bigr) ,\,\,\forall t>0.
	 	\end{align*}
	 	\item[(A2)] For each $\bb\in\mathcal{T}$, the limit of the free energy
	 	\begin{align*}
	 	F_N(\bb)&:=\frac{1}{N}\log \int_{\Sigma_N}\exp \sum_{p\geq 1}\beta_pH_{N,p}(\sigma)\nu_N(d\sigma)
	 	\end{align*}
	 	converges a.s. to some nonrandom $F(\bb)$. 
	 \end{itemize}
	
	\begin{remark}[Universality]\rm
	Throughout the paper, we neither impose any other assumptions on the Hamiltonians $(H_{N,p})$, nor on the measure space $(\Sigma_{N},\nu_{N})$. In particular, our results do not require $H_{N,p}$ to be a Gaussian process.
	\end{remark}
	
	 In the main examples that we are about to discuss in Subsection \ref{sec:Examples} below, the configuration space $\Sigma_{N}$ is either the discrete hypercube $\{\pm 1 \}^{N}$ or a sphere in $\mathbb R^{N}$ and the Hamiltoninans are the mixed $p$-spin models. They are known to satisfy assumptions (A0), (A1) and (A2).

	 Our first main theorem establishes a general duality principle that conjugates the temperature and energy parameters. Let
	 \begin{align*}
	 \mathcal{B}&=\bigg \{\bb=(\beta_p)_{p\geq 1}\in\mathcal{T}:\beta_p>0,\,\,\forall p\geq 1\bigg\}
	 \end{align*}
	 be the space of positive temperatures and
	 \begin{align*}
	 \mathcal{M}&=\bigg \{\bm=(m_p)_{p\geq 1}:m_p>0,\,\,\forall p\geq 1 \bigg \}
	 \end{align*} 
	 denote the space of scaled energies.
	 Define the free energy for the squared Hamiltonian,
	 \begin{align}\label{eq:VNdef}
	 V_N(\bm)&=\frac{1}{N}\log\int_{\Sigma_N}\exp  \sum_{p\geq 1}\frac{N}{2m_p}\Bigl(\frac{H_{N,p}(\sigma)}{N}\Bigr)^2\nu_N(d\sigma),\,\,\forall \bm\in \mathcal{M} .
	 \end{align} 

	 Denote by $V$ the limit of $(V_N)_{N\geq 1}$ whenever it exists in $\overline{\mathbb{R}}:=\mathbb R \cup \{\infty\}$. For any $\boldsymbol{\beta}\in\mathcal{B} $, $\bm\in\mathcal{M} $ and $f:[0,\infty)^{2}\rightarrow \mathbb{R}$, we set $f(\boldsymbol{\beta},\bm)=(f(\beta_p,m_p))_{p\geq 1}.$  We say that $F$ is G\^{a}teaux differentiable at $\bb$ if $$\lim_{t\rightarrow 0}\frac{F(\bb+t\bb')-F(\bb)}{t}$$ exists for any $\bb'\in\mathcal{T}.$ Our first result shows that $V$ exists and it can be expressed as a variational formula in terms of the free energy $F$. Moreover, if $F(\bb^{1/2})$ is concave in $\bb\in\mathcal{B} $, then this variational representation is invertible.
	 
	 \begin{theorem}[Legendre duality]\label{duality}
	 	If (A0), (A1) and (A2) hold, then we have
	 	\begin{itemize}
	 		\item[$(i)$] For any $\bm\in\mathcal{M} ,$
	 		\begin{align}
	 		\label{duality:eq1}
	 		V(\bm)&=\sup_{\bb\in\mathcal{B} }\Bigl(F(\bb)-\frac{1}{2}\sum_{p\geq 1} \beta_p^2 m_p\Bigr).
	 		\end{align}
	 		\item[$(ii)$] If $F(\bb)$ is G\^{a}teaux differentiable everywhere in $\mathcal{B} $ and $F(\bb^{1/2})$ is concave on $\mathcal{B} $, then for any $\bb\in \mathcal{B} ,$
	 		\begin{align}\label{duality:eq2}
	 		F(\bb)&=\inf_{\bm\in \mathcal{M} }\Bigl(V(\bm)+\frac{1}{2}\sum_{p\geq 1} \beta_p^2 m_p\Bigr).
	 		\end{align}
	 	\end{itemize}   	 
	 \end{theorem}
	 
	 Theorem \ref{duality} establishes the Legendre duality between $F$ and $V$, where the temperature and the scaled energy are conjugated variables. 
	 
	 \begin{remark}
	 	\rm
	 	It is clear that the variational representation \eqref{duality:eq2} is concave in $\bb^{1/2}$. A simple application of H\"{o}lder's inequality also implies the convexity of $F$. Such convexity can  be obtained from the variational formula of \eqref{duality:eq2}. Indeed, if we set $\boldsymbol{t}=(t_p)_{p\geq 1}$ for $t_p=\beta_p^2m_p$,
	 	\begin{align*}
	 	\inf_{\bm\in \mathcal{M} }\Bigl(V(\bm)+\frac{1}{2}\sum_{p\geq 1} \beta_p^2 m_p\Bigr)=\inf_{\boldsymbol{t}\in \mathcal{M} }\Bigl(U\Bigl(\frac{\bb^2}{\boldsymbol{t}}\Bigr)+\frac{1}{2}\sum_{p\geq 1} t_p\Bigr),
	 	\end{align*} 
	 	where $U(\bm):=V(\bm^{-1})$ for any $\bm\in\mathcal{M} .$
	 	For any $0<c<1$, $\bb,\bb'\in\mathcal{B} $ and $\boldsymbol{t},\boldsymbol{t}'\in\mathcal{M} $,  the Cauchy-Schwarz inequality says
	 	\begin{align*}
	 	\frac{((1-c)\beta_p+c\beta_p')^2}{(1-c)t_p+c t_p'}&\leq (1-c)\frac{\beta_p^2}{t_p}+c\frac{{\beta_p'}^2}{t_p'}.
	 	\end{align*}
	 	Observe that from \eqref{duality:eq1}, $U$ is nondecreasing in each coordinate and convex. This implies that
	 	\begin{align*}
	 	&U\Bigl(\frac{((1-c  )\bb+c  \bb')^2}{(1-c  )\boldsymbol{t}+c  \boldsymbol{t}'}\Bigr)+\frac{1}{2}\sum_{p\geq 1}( (1-c  )t_p+c   t_p')\\
	 	&\leq (1-c  )\Bigl(U\Bigl(\frac{\bb^2}{\boldsymbol{t}}\Bigr)+\frac{1}{2}\sum_{p\geq 1} t_p\Bigr)+c  \Bigl(U\Bigl(\frac{{\bb'}^2}{\boldsymbol{t}'}\Bigr)+\frac{1}{2}\sum_{p\geq 1} t_p'\Bigr),
	 	\end{align*}
	 	which clearly implies the convexity of $F$ by taking infimum on both sides.
	 \end{remark} 
	 
	 The next theorem describes the optimizers of the above two variational problems. Assume that $F$ is G\^{a}teaux differentiable on $\mathcal{B} $. Denote by $\partial_pF$ the partial derivative of $F$ with respect to $\beta_p.$ We define
	 \begin{align*}
	 \mathcal{B} _0&=\bigg \{\bb\in\mathcal{B} :\beta_p^{-1}\partial_pF(\bb)>0,\,\,\forall p\geq 1\bigg\}
	 \end{align*}
	 and
	 \begin{align*}
	 \mathcal{M} _0&=\bigg \{\bm\in\mathcal{M} :m_p=\beta_p^{-1}\partial_pF(\bb)>0,\,\,\forall p\geq 1\,\,\mbox{for some $\bb\in\mathcal{B} _0$}\bigg \}.
	 \end{align*} 
	 
	 \begin{theorem}[Optimality]\label{thm:op}
	 	Given the assumptions (A0), (A1) and (A2), if $F$ is G\^{a}teaux differentiable everywhere in $\mathcal{B}$ and $F(\bb^{1/2})$ is concave on $\mathcal{B} ,$ then we have
	 	\begin{itemize}
	 		\item[$(i)$] For any $\bm\in\mathcal{M} _0,$
	 		\begin{align}
	 		\label{thm:op:eq1}
	 		V(\bm)&=\max_{\bb\in\mathcal{B} _0}\Bigl(F(\bb)-\frac{1}{2}\sum_{p\geq 1} \beta_p^2 m_p\Bigr)
	 		\end{align}
	 		and any $\bb\in\mathcal{B} $ satisfying $m_p=\beta_p^{-1}\partial_pF(\bb)$ for all $p\geq 1$ is a maximizer.
	 		\item[$(ii)$] For any $\bb\in \mathcal{B} _0,$
	 		\begin{align}\label{thm:op:eq2}
	 		F(\bb)&=\min_{\bm\in \mathcal{M} _0}\Bigl(V(\bm)+\frac{1}{2}\sum_{p\geq 1} \beta_p^2 m_p\Bigr)
	 		\end{align}
	 		and $\bm\in\mathcal{M} $ defined by $m_p=\beta_p^{-1}\partial_pF(\bb)$ for all $p\geq 1$ is a minimizer.
	 	\end{itemize}   	 
	 \end{theorem}
	 
	 A word of comment is needed here. The dual parameter $\bm$ plays the role of a scaled energy.  If both free energies $F$ and $V$ are comparable, that is,  if the energies satisfy
	 $$ \beta_{p} H_{N,p} \sim \frac{1}{m_{p}} H_{N,p}^{2},$$
	 we can see that the parameter $m_{p}$ has the same scale as $H_{N,p}/\beta_{p}$. 
	 
	 Remark that although here we present results for the system with an infinite sequence of the Hamiltonians $(H_{N,p})_{p\geq 1}$, the statements of Theorems \ref{duality} and \ref{thm:op} remain valid if the system involves only a finite number of Hamiltonians, $(H_{N,p})_{1\leq p\leq s}$, for some fixed $s\in\mathbb{N}$ independent of $N.$ The case $s=1$ reads as follows.
	 
	 Assume that $H_N$ is a Hamiltonian indexed by the spin configuration space $\Sigma_N$ and is independent of $\nu_N$. Define the free energies by
	 \begin{align}
	 \begin{split}
	 \label{subsec:eq1}
	 F_N(\beta)&=\frac{1}{N}\log\int_{\Sigma_N}\exp\beta H_N(\sigma)\nu_N(d\sigma),\,\,\forall \beta\in\mathbb{R},\\
	 V_N(m)&=\frac{1}{N}\log \int_{\Sigma_N}\exp \frac{N}{2m}\Bigl(\frac{H_N(\sigma)}{N}\Bigr)^2\nu_N(d\sigma),\,\,\forall m>0.
	 \end{split}
	 \end{align}
	 Suppose that (A0), (A1) and (A2) are valid for $H_N.$ Denote by $F$ the limit of $(F_N).$ From Theorem \ref{duality}, the limit of $V_N$ exists, which is denoted by $V$. Theorem \ref{thm:op} becomes:
	 \begin{theorem}
	 	\label{thm:op:1}
	 	If $F(\beta)$ is differentiable and $F(\beta^{1/2})$ is concave in $\beta>0$, then for any $\beta\in \mathcal{B} _0:=\{\beta'>0:F'(\beta')>0\}$ and $m\in \mathcal{M} _0:=\{F'(\beta')/\beta':\mbox{for some $\beta'>0$}\},$ we have
	 	\begin{align}
	 	\begin{split}\label{thm5:eq2}
	 	V(m)&=\max_{\beta'\in \mathcal{B} _0}\Bigl(F(\beta')-\frac{{\beta'}^2m}{2}\Bigr),
	 	\end{split}\\
	 	\begin{split}\label{thm5:eq1}
	 	F(\beta)&=\min_{m'\in \mathcal{M} _0}\Bigl(V(m')+\frac{\beta^2m'}{2}\Bigr),
	 	\end{split}
	 	\end{align}
	 	where any $\beta'>0$ satisfying $m=F'(\beta')/\beta'$ is the maximizer of \eqref{thm5:eq2} and $m'=F'(\beta)/\beta$ is the minimizer of \eqref{thm5:eq1}.  
	 \end{theorem}

	 \subsection{Examples}\label{sec:Examples}
	 
	 We now discuss the Legendre structure established in the previous section for some mean-field spin glass models. They include the random energy model (REM) and the mixed $p$-spin models with both Ising and spherical spin configuration spaces. With the representations in Theorem \ref{thm:op}, we will show that the limiting free energy can be naturally written as a minimization problem that is related to Parisi's formulation of the same quantity.

\subsubsection{Random energy model}\label{sec:REM}

	 The REM model is defined on the configuration space $\Sigma_N:=\{-1,1\}^N$ and its Hamiltonian $H_N$ is a family of Gaussian process with covariance structure $\e H_N(\sigma^1)H_N(\sigma^2)=N\delta_{\sigma^1,\sigma^2}$ for all $\sigma^1,\sigma^2\in\Sigma_N,$ where $\delta_{\sigma^1,\sigma^2}$ equals $1$ if $\sigma^1=\sigma^2$ and $0$ otherwise. Let $\nu_N$ be the uniform probability measure on $\Sigma_N.$ It is well-known (see for instance \cite{T03}) that the conditions (A1) and (A2) hold and the limiting free energy has an explicit expression, for $\beta_c:=\sqrt{2\log 2},$
	 \begin{align}\label{eq5}
	 F(\beta)=\lim_{N\rightarrow\infty}\frac{1}{N}\log\sum_{\sigma}\frac{1}{2^N}\exp\beta H_N(\sigma) =
	 \left\{\begin{array}{ll}
	 \frac{\beta^2}{2},&\mbox{if $\beta\leq \beta_c$},\\
	 \\
	 \beta \beta_c-\log 2,&\mbox{if $\beta>\beta_c$}.
	 \end{array}\right.
	 \end{align}
	 Since $F(\beta^{1/2})$ is concave in $\beta>0,$ using Theorem \ref{thm:op:1}, one readily sees that $F$ and $V$ satisfy the Legendre duality with $\mathcal{B} _0=(0,\infty)$ and $\mathcal{M} _0=(0,1]$,
	 \begin{align*}
	 V(m)&=\max_{\beta>0}\Bigl(F(\beta)-\frac{{\beta}^2m}{2}\Bigr),\,\,0<m\leq 1,\\
	 F(\beta)&=\min_{0<m\leq 1}\Bigl(V(m)+\frac{\beta^2 m}{2}\Bigr),\,\,\beta>0.
	 \end{align*}
	 Note that a direct computation gives $V(m)=\log 2/m-\log 2$ for $0<m\leq 1$. Plugging this into the variational formula for $F$, we can express the limiting free energy of the REM in Parisi's formulation, previously derived in \cite{BK,G13}:

	 \begin{theorem}[Parisi formula for the REM]
	 	We have that
	 	\begin{align*}
	 	F(\beta)&=\inf_{0<m\leq 1}\Bigl(\frac{\beta^2m}{2}+\frac{\log 2}{m}-\log 2\Bigr).
	 	\end{align*}
	 \end{theorem}

	 	 Remarkably, the same argument also yields the Parisi formula for the REM model with random external field, see \cite{AK}.

	 \subsubsection{Ising mixed $p$-spin model}\label{sec:Isingpspin}
	 
	 The Ising mixed $p$-spin model is defined on the same configuration space as the REM. The Hamiltonian $H_N$ is a centered Gaussian process indexed by $\Sigma_N$ and has the covariance structure, $$\e H_N(\sigma^1)H_N(\sigma^2)=N\xi(R_{1,2})$$
	 for $\sigma^1,\sigma^2\in\Sigma_N,$ where $\xi(s):=\sum_{p\geq 2}c_p^2s^p$ with $\sum_{p\geq 2}2^pc_p^2<\infty$ and $R_{1,2}:=N^{-1}\sum_{1\leq i\leq N}\sigma_i^1\sigma_i^2$ is the overlap between $\sigma^1$ and $\sigma^2.$ Using the same definitions as \eqref{subsec:eq1}, the conditions (A0), (A1) and (A2) are valid  for $H_N$. In particular, it is famously known that the limiting free energy can be computed through Parisi's formula (see \cite{P13,T11}), 
	 \begin{align}\label{eq:add}
	 F(\beta)&=\lim_{N\rightarrow\infty}\frac{1}{N}\log \sum_{\sigma}\frac{1}{2^N}\exp \beta H_N(\sigma)=\inf_{\alpha\in \mathcal{D}}\mathcal{P}_\beta(\alpha),
	 \end{align}
	 where letting $\mathcal{D}$ be the space of all probability distribution functions on $[0,1]$, $\mathcal{P}_\beta$ is defined as
	 \begin{align*}
	 \mathcal{P}_\beta(\alpha)&=\Phi_{\alpha,\beta}(0,0)-\frac{\beta^2}{2}\int_0^1\alpha(s)s\xi''(s)ds
	 \end{align*}
	 for $\Phi_{\alpha,\beta}$ being the solution to the Parisi PDE with boundary condition $\Phi_{\alpha,\beta}(1,x)=\log \cosh(x)$,
	 $$
	 \partial_s\Phi_{\alpha,\beta}=-\frac{\beta^2\xi''(s)}{2}\Bigl(\partial_{xx}^2\Phi_{\alpha,\beta}+\alpha(s)(\partial_x\Phi_{\alpha,\beta})^2\Bigr),\,\,(s,x)\in[0,1)\times\mathbb{R}.
	 $$
	 It is proven in \cite{AC14} that the Parisi formula has a unique minimizer, denoted by $\alpha_{P,\beta}.$ Set $\mathcal{D}_0$ the collection of all $\alpha_{P,\beta}$ associated to any temperatures $\beta>0.$  It was recently proved in \cite{AC15} that $F(\beta^{1/2})$ is a concave function in $\beta>0.$ As one can compute the derivative of $F$ from the Parisi formula (see \cite{P08}),
	  \begin{align}
	  \label{prop2:proof:eq2}
	  F'(\beta)&=\beta\int_0^1\alpha_{P,\beta}(s)\xi'(s)ds>0,\,\,\forall \beta>0,
	  \end{align}
	  it follows that the duality in Theorem \ref{thm:op:1}  holds, where $\mathcal{B}_0=(0,\infty)$ and $$\mathcal{M}_0=\bigg \{\int_0^1\alpha(s)\xi'(s)ds:\alpha\in \mathcal{D}_0\bigg \}.$$ Define the Legendre transform of $\mathcal{P}_\beta$ by
	  	\begin{align}\label{eq-1}
	  	\Gamma(\alpha)&=\sup_{\beta>0}\Bigl(\mathcal{P}_\beta(\alpha)-\frac{\beta^2}{2}\int_0^1\alpha(s)\xi'(s)ds\Bigr),\,\,\forall \alpha\in\mathcal{D}.
	  	\end{align}
	  Our result below says that the limiting free energy of the squared Hamiltonian $V$ can be identified as $\Gamma$, from which it allows to rewrite the duality obtained from Theorem \ref{thm:op:1} in terms of $\Gamma.$ This alternative expression was firstly conjectured by Guerra \cite{G15} and established in \cite{AC15}.
	  
	  \begin{theorem}[Legendre structure]\label{thm7}
	  	Given $\alpha\in\mathcal{D}_0$, $m:=\int_0^1\alpha(s)\xi'(s) ds$, we have that
	  		\begin{align}\label{prop2:eq2}
	  		V(m)&=\Gamma(\alpha).
	  		\end{align}
	  Moreover, for any $\alpha\in\mathcal{D}_0$ and $\beta\in\mathcal{B}_0$,  
	  	\begin{align}
	  	\begin{split}\label{thm7:eq1}
	    \Gamma(\alpha)&=\max_{\beta'\in\mathcal{B}_0}\Bigl(F(\beta')-\frac{{\beta'}^2}{2}\int_0^1\alpha(s)\xi'(s)ds\Bigr),
	    \end{split}\\
	    \begin{split}\label{thm7:eq2}
	  	F(\beta)&=\min_{m'\in\mathcal{M}_0}\Bigl(\Gamma(\alpha')+\frac{\beta^2}{2}\int_0^1\alpha'(s)\xi'(s)ds\Bigr),
	  	\end{split}
	  	\end{align}
	  	where the maximizer of the first variational principle is equal to $\beta$ if $\alpha=\alpha_{P,\beta}$ and the minimizer of the second one is given by $\alpha_{P,\beta}$.
	  \end{theorem}
	 	

	 \begin{proof}
	    First we verify \eqref{prop2:eq2}. Let $m=\int_0^1\alpha(s)\xi'(s) ds$, where $\alpha=\alpha_{P,\beta}$ for some $\beta>0.$ Note that from \eqref{thm5:eq2}, the characterization of the maximizer of \eqref{thm5:eq2} and the derivative of $F$ in \eqref{prop2:proof:eq2}, we have
	    \begin{align}
	    \begin{split}\label{thm7:proof:eq3}
	    V(m)&=\max_{\beta'\in\mathcal{B}_0}\Bigl(F(\beta')-\frac{{\beta'}^2}{2}\int_0^1\alpha(s)\xi'(s)ds\Bigr)\\
	    &=F(\beta)-\frac{\beta^2}{2}\int_0^1\alpha(s)\xi'(s)ds.
	    \end{split}
	    \end{align}
	    To see that this matches $\Gamma(\alpha)$, we shall need a few properties of the Parisi functional $\mathcal{P}_\beta(\alpha)$ previously established in \cite{C14,AC15}. First, from \cite[Equation (13)]{AC15}, for any $\beta'>0,$
	 	\begin{align}\label{prop2:proof:eq1}
	 	\frac{d}{d\beta'}\Bigl(\mathcal{P}_{\beta'}(\alpha)-\frac{{\beta'}^2}{2}\int_0^1\alpha(s)\xi'(s)ds\Bigr)&=-\beta'\int_0^1\xi'(s)\bigl(\e u_{\alpha,\beta'}(s)^2-s\bigr)d\alpha(s)\Bigr),
	 	\end{align}
	 	where $u_{\alpha,\beta'}$ is a uniformly bounded progressively measurable process with respect to the filtration generated by a standard Brownian motion.
	 	From \cite[Equation $(28)$]{AC15}, it is understood that for any $s\in[0,1],$ $\e u_{\alpha,\beta'}(s)^2$ is nondecreasing in $\beta'$. In addition, from \cite[Proposition~$1$]{C14}, we also have that $\e u_{\alpha,\beta'}(s)^2=s$ for any $s$ in the support of $\alpha$ if $\beta'=\beta$. As a result, the derivative \eqref{prop2:proof:eq1} is $\geq 0$ for $\beta'<\beta$ and $\leq 0$ if $\beta'>\beta.$ From the definition \eqref{eq-1} of $\Gamma(\alpha)$ and noting that $\alpha$ is the minimizer in Parisi's formula, this concludes 
	 	\begin{align*}
	 	\Gamma(\alpha)&=\mathcal{P}_\beta(\alpha)-\frac{\beta^2}{2}\int_0^1\alpha(s)\xi'(s)ds,
	 	\end{align*} 
	    which combined with \eqref{thm7:proof:eq3} finishes the proof of \eqref{prop2:eq2}.
	 
	 	Next, note that \eqref{thm7:eq1} follows straightforwardly from \eqref{prop2:eq2} and \eqref{thm7:proof:eq3}, while \eqref{thm7:eq2} can be obtained from \eqref{thm5:eq1} and \eqref{prop2:eq2},
	 		\begin{align*}
	 		F(\beta)&=\inf_{m\in\mathcal{M}_0}\Bigl(V(m)+\frac{\beta^2m}{2}\Bigr)\\
	 		&=\inf_{\beta'>0}\Bigl(\Gamma(\alpha_{P,\beta'})+\frac{\beta^2}{2}\int_0^1\alpha_{P,\beta'}(s)\xi'(s)ds\Bigr)\\
	 		&=\inf_{\alpha\in\mathcal{D}_0}\Bigl(\Gamma(\alpha)+\frac{\beta^2}{2}\int_0^1\alpha(s)\xi'(s)ds\Bigr)
	 		\end{align*}
	 	where the minimizer is obtained by  for $\alpha=\alpha_{P,\beta}$. This completes our proof.
	 \end{proof}

	 \subsubsection{Spherical mixed $p$-spin model}
	 
	 Let $\Sigma_N=\{\sigma\in\mathbb{R}^N:\sum_{i=1}^N\sigma_i^2=N\}$ and $\nu_N$ be a uniform probability measure on $\Sigma_N.$ For every $p\geq 1,$ set the pure $p$-spin Hamiltonian as
	 \begin{align*}
	 H_{N,p}(\sigma)&=\frac{1}{2^{p/2}N^{(p-1)/2}}\sum_{1\leq i_1,\ldots,i_p\leq N}g_{i_1,\ldots,i_p}\sigma_{i_1}\cdots\sigma_{i_p},
	 \end{align*}
	 where $g_{i_1,\ldots,i_p}$'s are i.i.d. standard Gaussian for all $1\leq i_1,\ldots,i_p\leq N$ and $p\geq 1.$ Throughout this subsection, we shall use the notations defined in Subsection \ref{sec1.1} with this family of Hamiltonians $(H_{N,p})_{p\geq 1}$. In this case, the assumption (A1) holds by standard Gaussian concentration of measure for Lipschitz functions and Dudley's entropy integral (see for instance \cite[Appendix]{ArnabChen15}). Assumption (A2) holds and the limiting free energy $F(\bb)$ can be computed by a spherical version of the Parisi formula \cite{TalagrandSphh,C13}. Here, Parisi's formula has a simpler expression, first derived by Crisanti-Sommers \cite{CS}. It reads
	 \begin{align}\label{spf}
	 F(\bb)=\inf_{\alpha\in\mathcal{C}} \mathcal{Q}_{\bb}(\alpha),
	 \end{align}
	 where $\mathcal{C}$ is the space of all distribution functions $\alpha$ on $[0,1]$ with $\alpha(\hat{q})=1$ for some $\hat{q}<1$ and setting $$\xi_{\bb}(q)=\sum_{p\geq 1}2^{-p}\beta_p^2q^p,$$ the functional $\mathcal Q_{\bb}$ is defined by
	 \begin{align*}
	 \mathcal{Q}_{\bb}(\alpha)&=\frac{1}{2}\Bigl(\int_0^1\alpha(q)\xi_{\bb}'(q)dq+\int_0^{\hat{q}}\frac{dq}{\int_q^1\alpha(s)ds}+\log(1-\hat{q})\Bigr).
	 \end{align*}
	 Note that for any fixed $\alpha$, this quantity is independent of the choice of $\hat{q}<1$ with $\alpha(\hat{q})=1.$ As this functional $\mathcal{Q}_{\bb}$ is strictly convex, it ensures that the optimization problem \eqref{spf} has a unique minimizer, denoted by $\alpha_{P,\bb}.$ If now one replaces $\bb$ by $\bb^{1/2},$ then $\mathcal{Q}_{\bb^{1/2}}(\alpha)$ is linear in $\bb\in\mathcal{B}$ and as a result, $F(\bb^{1/2})$ is concave on $\mathcal{B}.$ Consequently, knowing the fact that $$
	 \partial_pF(\bb)=\frac{p\beta_p}{2^p}\int_0^1\alpha_{P,\bb}(q)q^{p-1}dq>0
	 $$
	 as long as $\beta_p>0$ (see \cite{TalagrandSphh}), it implies that $\mathcal{B}_0=\mathcal{B}$ and
	 \begin{align*}
	 \mathcal{M}_0&=\Bigl\{\bm:m_p=\frac{p}{2^p}\int_0^1\alpha_{P,\bb}(q)q^{p-1}dq,\,\,\forall p\geq 1,\bb\in\mathcal{B}\Bigr\}.
	 \end{align*} Using Theorem \ref{thm:op}, we then obtain the Legendre duality between $V$ and $F.$ Let $\mathcal{C}_0$ be the collection of all $\alpha_{P,\bb}$ for some $\bb\in\mathcal{B}.$ Consider the Legendre transform of $\mathcal{Q}_{\bb}$,
	 \begin{align*}
	 {\Lambda}(\alpha)&:=\sup_{\bb\in\mathcal{B}}\Bigl(\mathcal{Q}_{\bb}(\alpha)-\frac{1}{2}\int_0^1\alpha(q)\xi_{\bb}'(q)dq\Bigr)\\
	 &=\int_0^{\hat{q}}\frac{dq}{\int_q^1\alpha(s)ds}+\log(1-\hat{q})
	 \end{align*}
	 for any $\alpha\in\mathcal{C}.$ We show that the free energy associated to the squared Hamiltonians $V$ is equal to $\Lambda$. This and the duality in Theorem \ref{thm:op} together imply that the Crisanti-Sommers formula can be regarded as the Legendre transform of $V.$ More precisely, below is the statement of our main result in this subsection.
	 
	 \begin{theorem}\label{thm10}
	 For any $\bm$ with $$
	 m_p=\frac{p}{2^p}\int_0^1\alpha(q)q^{p-1}dq,\,\,\forall p\geq 1,
	 $$ for some $\alpha\in\mathcal{C}_0$, we have
	 \begin{align}\label{thm10:eq-1}
	 V(\bm)&=\Lambda(\alpha).
	 \end{align}
	  Moreover, for any $\alpha\in\mathcal{C}_0$ and $\bb\in\mathcal{B}_0$,  
	  \begin{align}
	  \begin{split}\label{thm10:eq1}
	  \Lambda(\alpha)&=\max_{\bb'\in\mathcal{B}_0}\Bigl(F(\bb')-\frac{1}{2}\int_0^1\alpha(s)\xi_{\bb'}'(s)ds\Bigr),
	  \end{split}\\
	  \begin{split}\label{thm10:eq2}
	  F(\bb)&=\min_{\alpha'\in\mathcal{C}_0}\Bigl(\Lambda(\alpha')+\frac{1}{2}\int_0^1\alpha'(s)\xi_{\bb}'(s)ds\Bigr),
	  \end{split}
	  \end{align}
	  where the maximizer of the first variational principle is equal to $\bb$ if $\alpha=\alpha_{P,\bb}$ and the minimizer of the second one is given by $\alpha_{P,\bb}$.
	 \end{theorem}
	 
	 \begin{proof} Let $\bm\in\mathcal{M}_0$ satisfy $
	 	m_p=2^{-p}p\int_0^1\alpha(q)q^{p-1}dq$ for all $p\geq 1$ and some $\alpha\in\mathcal{C}_0$. Assume that $\alpha=\alpha_{P,\bb}$ for $\bb\in\mathcal{B}.$ Note that $m_p=\beta_p^{-1}\partial_pF(\bb).$ From \eqref{thm:op:eq1} and the characterization of the maximizer, we have that
	 	\begin{align*}
	 	V(\bm)&=\max_{\bb'\in\mathcal{B}_0}\Bigl(F(\bb')-\frac{1}{2}\sum_{p\geq 1}{\beta_p'}^2m_p\Bigr)\\
	 	&=F(\bb)-\frac{1}{2}\sum_{p\geq 1}\beta_p^2m_p\\
	 	&=F(\bb)-\frac{1}{2}\int_0^1\alpha(q)\xi_{\bb}'(q)dq\\
	 	&=\Lambda(\alpha),
	 	\end{align*}
	 	where the last equation used the Crisanti-Sommers formula \eqref{spf} and the definition of $\alpha_{P,\bb}.$ This gives \eqref{thm10:eq-1} and evidently also \eqref{thm10:eq1}, while \eqref{thm10:eq2} follows directly from \eqref{thm:op:eq2} and  \eqref{thm10:eq-1}. The characterizations of the optimizers of \eqref{thm10:eq1} and \eqref{thm10:eq2} are valid from Theorem~\ref{thm:op}.
	 \end{proof}
 	 
	 \subsection{The Parisi formula as a maximization problem}
	 In this subsection, we return to the Ising mixed $p$-spin model, considered in Subsection \ref{sec:Isingpspin} and show one application of Theorem \ref{thm7}. Recall $F(\beta)$ from \eqref{eq:add} and $\Gamma(\alpha)$ from \eqref{eq-1}. 
	 Consider the space $\mathbf B$ of bounded functions on $[0,1]$. For $f \in \mathbf B$, we define the following two real-valued functionals,

\begin{equation*}
\begin{split}
L_{*}(f)&= \inf_{\alpha \in \mathcal D}\Bigl( \int_{0}^{1} \alpha(s)f(s)ds    +\frac{\beta^{2}}{2}\int_0^1\alpha(s)\xi'(s)ds \bigg )
\end{split}
\end{equation*}
and
$$\Gamma^{*}(f)=\sup_{\alpha \in \mathcal D} \bigg (\int_{0}^{1} \alpha(s)f(s)ds  - \Gamma(\alpha) \bigg ).$$
The following corollary is a combination of Theorem \ref{thm7} and convex duality.

\begin{corollary}\label{Parisiisatruesup} For any $\beta \geq 0$,
\begin{equation}\label{eq:theoremequation}
F(\beta) = \max_{f \in \mathbf B} \bigl(L_{*}(f)-\Gamma^{*}(f)  \bigr).
\end{equation}

\end{corollary}

The variational problem \eqref{eq:theoremequation} mimicks a maximizer coming from an entropy-energy variational principle. First, only $L_{*}$ depends on the temperature parameter. Second, the functional inside the variational principle defining $L_{*}$ is linear and the function $\Gamma^{*}$ is convex and lower semi-continuous.  We comment and derive further properties of \eqref{eq:theoremequation} at the end of Section \ref{proofSec2}.

\begin{proof}[\bf Proof of Corollary \ref{Parisiisatruesup}]
Recall that from Theorem \ref{thm7} and the Parisi formula, we can alternatively rewrite \eqref{thm7:eq2} as 
\begin{align*}
F(\beta)&=\inf_{\alpha \in \mathcal D}\Bigl(\Gamma(\alpha)+\frac{\beta^2}{2}\int_0^1\alpha(s)\xi'(s)ds\Bigr).
\end{align*}	 
Since
\begin{align*}
\Gamma(\alpha)+\frac{\beta^{2}}{2}\int_0^1\alpha(s)\xi'(s)ds&=\Bigl(\int_0^1\alpha(s)f(s)ds+\frac{\beta^2}{2}\int_0^1\alpha(s)\xi'(s)ds\Bigr)\\
&\qquad+ \Bigl(\Gamma(\alpha)-\int_0^1\alpha(s)f(s)ds\Bigr),
\end{align*}
we have that for all $f \in \mathbf B$,
\begin{align*}
\Gamma(\alpha)+\frac{\beta^2}{2}\int_0^1\alpha(s)\xi'(s)ds&\geq L_*(f)-\Gamma^*(f)
\end{align*}
and thus,
\begin{align*}
F(\beta)&\geq \sup_f\bigl(L_*(f)-\Gamma^*(f)\bigr).
\end{align*}
To see the reverse inequality, we choose $f_{\beta}(s)=\beta^2\xi'(s)/2$ for $s\in[0,1].$ Then $L_*(f_\beta)=0$ and obviously the Parisi formula says
$\Gamma^*(f_{\beta})=-F(\beta).$ From these,
\begin{align*}
\sup_f\bigl(L_*(f)-\Gamma^*(f)\bigr)&\geq F(\beta)
\end{align*}
and so \eqref{eq:theoremequation} follows.

\end{proof}

	 \subsection{Proofs of Theorems \ref{duality} and \ref{thm:op} } \label{proofSec2}
	 
	 The idea of establishing the variational formula for $V$ in Theorem \ref{duality} is to approximate the squared Hamiltonians using linear approximations. One of the main difficulties here comes from the fact that the interaction involves infinitely many components. This is the reason why we need assumption (A1), which will be used to control all  Hamiltonians at once. More precisely, the following lemma gives uniform bounds for $H_{N,p}$ on $\Sigma_N$ for all $N,p\geq 1$ if the condition (A1) is given.

	 \begin{lemma}\label{add} Given  assumption (A1), there exists a sequence of positive real numbers $(b_p)_{p\geq 1}$ with $b_p<C_12^{-p/4}$ such that
	 	\begin{align}
	 	\label{add:eq1}
	 	\p\Bigl(\exists N_0\geq 1\,\,\mbox{such that}\,\,\sup_{\sigma\in\Sigma_N}\frac{|H_{N,p}(\sigma)|}{N}< b_p,\,\,\forall p\geq 1,N\geq N_0\Bigr)=1.
	 	\end{align}
	 	
	 \end{lemma}
	 
	 \begin{proof}
	 	From (A1), we have that
	 	\begin{align*}
	 	\e \sup_{\sigma\in\Sigma_N}\frac{|H_{N,p}(\sigma)|}{N}\leq \frac{C_1}{2^{p/2}}
	 	\end{align*}
	 	and that with probability at most $C_2e^{-2^pNt^2/C_2}$, the following event holds,
	 	\begin{align*}
	 	\sup_{\sigma\in\Sigma_N}\frac{|H_{N,p}(\sigma)|}{N}\geq \e \sup_{\sigma\in\Sigma_N}\frac{|H_{N,p}(\sigma)|}{N}+t.
	 	\end{align*}
	 	Combining these together, with probability at most $C_2e^{-2^pNt^2/C_2},$
	 	\begin{align*}
	 	\sup_{\sigma\in\Sigma_N}\frac{|H_{N,p}(\sigma)|}{N}\geq t-\frac{C_1}{2^{p/2}}.
	 	\end{align*}
	 	In particular, taking $t={2^{-p/4}}C_1$, we obtain that with probability at most
	 	$C_2e^{-2^{p/2}NC_1^2/C_2}
	 	$,
	 	\begin{align*}
	 	\sup_{\sigma\in\Sigma_N}\frac{|H_{N,p}(\sigma)|}{N}\geq b_p:=\frac{C_1}{2^{p/4}}\Bigl(1-\frac{1}{2^{p/4}}\Bigr).
	 	\end{align*}
	 	Denote by $\Omega_N$ the event that there exists some $p\geq 1$ such that 
	 	$$
	 	\sup_{\sigma\in\Sigma_N}\frac{|H_{N,p}(\sigma)|}{N}\geq b_p.
	 	$$
	 	Since
	 	\begin{align*}
	 	\sum_{N\geq 1}\p(\Omega_N)&\leq \sum_{N\geq 1}\sum_{p\geq 1}\p\Bigl(\sup_{\sigma\in\Sigma_N}\frac{|H_{N,p}(\sigma)|}{N}\geq b_p\Bigr)\\
	 	&\leq C_2\sum_{N\geq 1}\sum_{p\geq 1}e^{-2^{p/2}NC_1^2/C_2}<\infty,
	 	\end{align*}
	 	the Borel-Cantelli lemma says that the event that $\Omega_N$ occurs infinitely often has zero probability. This gives the announced result.
	 \end{proof}

	 For each $k\ge 1,$ set
	 \begin{align*}
	 \mathcal{B} _k&=\{\bb_k=(\beta_p)_{1\leq p\leq k}:\beta_p>0,\,\,\forall 1\leq p\leq k\},\\
	 \mathcal{M} _k&=\{\bm_k=(m_p)_{1\leq p\leq k}:m_p>0,\,\,\forall 1\leq p\leq k\}.
	 \end{align*}
	 Consider the following free energies,
	 \begin{align*}
	 F_{N,k}(\bb_k)&=\frac{1}{N}\log \int_{\Sigma_N}\exp \sum_{p=1}^k\beta_pH_{N,p}(\sigma)\nu_N(d\sigma),\,\,\forall \bb_k\in\mathcal{B} _k,\\
	 V_{N,k}(\bm_k)&=\frac{1}{N}\log \int_{\Sigma_N}\frac{1}{2^N}\exp \sum_{p=1}^k\frac{N}{2m_p}\Bigl(\frac{H_{N,p}(\sigma)}{N}\Bigr)^2\nu_N(d\sigma),\,\,\forall \bm_k\in \mathcal{M} _k.
	 \end{align*}
	 We use $F_k$ and $V_k$ to denote the limits of $(F_{N,k})_{N\geq 1}$ and $(V_{N,k})_{N\geq 1}$ whenever they exist in $\overline{\mathbb{R}}$.

	 \begin{lemma}\label{add2}
	 	Assume that (A0), (A1) and (A2) hold. Consider any $I\subseteq\mathbb{N}$ and $(\beta_p)_{p\notin I}$ with $\sum_{p\notin I}\beta_p^2<\infty.$ For any $\bg=(\gamma_p)_{p\in I}$ with $\sum_{p\in I}\gamma_p^2<\infty$, define the restriction of $F$ to $I$ by $$g(\bg)=F(\phi(\bg)),$$ where $\phi(\bg):=(\delta_p)_{p\geq 1}$ satisfies $\delta_p=\gamma_p$ if $p\in I$ and $\delta_p=\beta_p$ if $p\notin I.$ Then $g$ is continuous at $\boldsymbol{0}:=(0,0,\ldots).$
	 \end{lemma}
	 
	 \begin{proof}
	 	From the convexity of $F_{N}$ on the set of coordinates $I$,
	 	\begin{align}
	 	\label{add2:proof:eq1}
	 	-\sum_{p\in I}\partial_pF_{N}(\phi(\bg))\gamma_p\leq F_{N}(\phi(\bg))-F_{N,k}(\phi(\boldsymbol{0}))\leq \sum_{p\in I}\partial_pF_{N}(\phi(\boldsymbol{0}))\gamma_p
	 	\end{align}
	 	for any $\gamma=(\gamma_p)_{p\in I}$ with $\sum_{p\in I}\gamma_p^2<\infty.$ To control the two sides, observe that 
	 	\begin{align}\label{add2:proof:eq2}
	 	\inf_{\sigma\in\Sigma_N}\frac{H_{N,p}(\sigma)}{N}\leq\partial_p F_{N}(\phi(\bg))&=\Bigl\la\frac{H_{N,p}(\sigma)}{N}\Bigr\ra_{\bg}\leq \sup_{\sigma\in\Sigma_N}\frac{H_{N,p}(\sigma)}{N},\,\,\forall p\in I,
	 	\end{align}
	 	where $\la\cdot\ra_{\bg}$ is the expectation with respect to the Gibbs measure, $$
	 	G_{\bg}(d\sigma)=\frac{\exp\Bigl(\sum_{p\notin I}\beta_pH_{N,p}(\sigma)+\sum_{p\in I}\gamma_pH_{N,p}(\sigma)\Bigr)\nu_N(d\sigma)}{\int_{\Sigma_N}\exp\Bigl(\sum_{p\notin I}\beta_pH_{N,p}(\sigma')+\sum_{p\in I}\gamma_pH_{N,p}(\sigma')\nu_N(\sigma') }.$$
	 	From the condition (A0), the random variables $-\inf_{\sigma\in\Sigma_N}H_{N,p}(\sigma)$ and $\sup_{\sigma\in\Sigma_N}H_{N,p}(\sigma)$ have the same distribution.
	 	Taking expectations in \eqref{add2:proof:eq1} and applying \eqref{add2:proof:eq2} lead to
	 	\begin{align*}
	 	|\e F_{N}(\phi(\bg))-\e F_N(\phi(\boldsymbol{0}))|&\leq C_1\sum_{p\in I}\frac{\gamma_p}{2^{p/2}},
	 	\end{align*}
	 	where the inequality used (A1).
	 	Note that the condition (A2) implies $\lim_{N\rightarrow\infty} \e F_{N}(\phi(\bg))= F(\phi(\bg))$ by the dominated convergence theorem. Hence, passing to the limit in the last inequality gives
	 	\begin{align*}
	 	|g(\bg)-g(\boldsymbol{0})|&\leq C_1\sum_{p\in I}\frac{\gamma_p}{2^{p/2}}.
	 	\end{align*}
	 	The continuity of $F$ at $\boldsymbol{0}$ follows evidently.
	 \end{proof}
	 
	 The next lemma  is a finite dimensional version of \eqref{duality:eq1}.
	 
	 \begin{lemma}
	 	\label{lem12}
	 	Suppose that (A0), (A1) and (A2) hold. For any $\bm_k\in\mathcal{M} _k,$ we have
	 	\begin{align*}
	 	V_k(\bm_k)&=\sup_{\bb_k\in \mathcal{B} _k}\Bigl(F_k(\bb_k)-\frac{1}{2}\sum_{p=1}^k\beta_p^2 m_p\Bigr).
	 	\end{align*}
	 \end{lemma}
	 
	 \begin{proof}
	    From Lemma \ref{add}, there exists a sequence of positive real numbers $(b_p)_{p\geq 1}$ such that with probability one, there exists some $N_0>0$ such that 
	 	\begin{align*}
	 	\frac{H_{N,p}(\sigma)}{N}\in[-b_p,b_p],\,\,\forall \sigma\in\Sigma_N\,\,\mbox{and}\,\,\forall p\geq 1, N\geq N_0,
	 	\end{align*}
	 	Since we only care about the limiting behavior of $(V_{k,N})_{N\geq 1}$, we assume without loss of generality that $|H_{N,p}/N|\leq b_p$ for all $1\leq p\leq k$ and $N\geq 1.$ For any $\varepsilon>0,$ let $\ell_p\in\mathbb{N}$ with $\ell_p\geq b_p/\varepsilon.$ Set $\beta _{p,i}=i\varepsilon/m_p$ for $i\in [-\ell_p,\ell_p]\cap\mathbb{Z}.$ For any $i_p\in [-\ell_p,\ell_p]\cap\mathbb{Z}$ for $1\leq p\leq k,$ define
	 	\begin{align*}
	 	A_{p,i_1,\ldots,i_k}&=\Bigl\{\sigma\in\Sigma_N\Big|\frac{H_{N,p}(\sigma)}{N}\in[(i_p-1)\varepsilon,(i_p+1)\varepsilon],\,\,\forall 1\leq p\leq k\Bigr\}.
	 	\end{align*}
	 	Observe that for any $\sigma\in A_{p,i_1,\ldots,i_k}$, since $\beta _{p,i_p}m_p=i_p\varepsilon$, we have
	 	\begin{align*}
	 	\sum_{p=1}^k\frac{N}{2m_p}\Bigl(\frac{H_{N,p}(\sigma)}{N}\Bigr)^2
	 	&=\sum_{p=1}^k\frac{N}{2m_p}\Bigl(\frac{H_{N,p}(\sigma)}{N}-\beta _{p,i_p}m_p+\beta _{p,i_p}m_p\Bigr)^2\\
	 	&= \sum_{p=1}^k\frac{N}{2m_p}\Bigl(\frac{H_{N,p}(\sigma)}{N}-\beta _{p,i_p}m_p\Bigr)^2\\
	 	&\,\,+\sum_{p=1}^kN\beta _{p,i_p}\Bigl(\frac{H_{N,p}(\sigma)}{N}-\beta _{p,i_p}m_p\Bigr)+
	 	\sum_{p=1}^k\frac{N\beta _{p,i_p}^2m_p}{2}\\
	 	&\leq N\varepsilon^2\sum_{p=1}^k\frac{1}{2m_p}+\sum_{p=1}^k\beta _{p,i_p}H_{N,p}(\sigma)-N\sum_{p=1}^k\frac{\beta _{p,i_p}^2m_p}{2},
	 	\end{align*}
	 	from which
	 	\begin{align*}
	 	&\int_{\Sigma_N}\exp \sum_{p=1}^k\frac{N}{2m_p}\Bigl(\frac{H_{N,p}(\sigma)}{N}\Bigr)^2\nu_N(d\sigma)\\
	 	&\leq \sum_{|i_1|\leq\ell_1,\ldots,|i_k|\leq\ell_k}\exp\Bigl(-N\sum_{p=1}^k\frac{\beta _{p,i_p}^2m_p}{2}+N\varepsilon^2\sum_{p=1}^k\frac{1}{2m_p}\Bigr)\\
	 	&\qquad\qquad\qquad\cdot\int_{\Sigma_N}1_{A_{p,i_1,\ldots,i_k}}(\sigma)\exp\Bigl(\sum_{p=1}^k\beta _{p,i_p}H_{N,p}(\sigma)\Bigr)\nu_N(d\sigma)\\
	 	&\leq \sum_{|i_1|\leq\ell_1,\ldots,|i_k|\leq\ell_k}\exp\Bigl(-N\sum_{p=1}^k\frac{\beta _{p,i_p}^2m_p}{2}+N\varepsilon^2\sum_{p=1}^k\frac{1}{2m_p}\Bigr)\\
	 	&\qquad\qquad\qquad\cdot\int_{\Sigma_N}\exp\Bigl(\sum_{p=1}^k\beta _{p,i_p}H_{N,p}(\sigma)\Bigr)\nu_N(d\sigma).
	 	\end{align*}
	 	Consequently, taking $N^{-1}\log$ on both sides leads to
	 	\begin{align*}
	 	V_{N,k}(\bm_k)&\leq \frac{\log\prod_{p=1}^k(2\ell_p+1)}{N}+\varepsilon^2\sum_{p=1}^k\frac{1}{2m_p}\\
	 	&\,\,+\max_{-\ell_p\leq i_p\leq \ell_p,\,\,\forall 1\leq p\leq k}\Bigl(F_{N,k}(\beta_{k,i_1},\ldots,\beta_{k,i_p})-\frac{1}{2}\sum_{p=1}^k\beta _{p,i}^2m_p\Bigr)
	 	\end{align*}
	 	and from (A2), passing to limit gives
	 	\begin{align*}
	 	\limsup_{N\rightarrow\infty}V_{N,k}(\bm_k)&\leq \varepsilon^2\sum_{p=1}^k\frac{1}{m_p}+\max_{-\ell_p\leq i_p\leq \ell_p,\,\,\forall 1\leq p\leq k}\Bigl(F_{k}(\beta_{k,i_1},\ldots,\beta_{k,i_p})-\frac{1}{2}\sum_{p=1}^k\beta _{p,i}^2m_p\Bigr)\\
	 	&\leq\varepsilon^2\sum_{p=1}^k\frac{1}{m_p}+ \sup_{\bb_k\in\mathcal{B} _k}\Bigl(F_{k}(\bb _k)-\frac{1}{2}\sum_{p=1}^k\beta _p^2m_p\Bigr),
	 	\end{align*}
	 	where the last inequality used the fact that $F_k$ is an even function in each coordinate and $F_k$ is continuous at $\boldsymbol{0}_k=(0,\ldots,0)$ by applying Lemma \ref{add2}.
	 	Since this holds for any $\varepsilon,$ it follows
	 	\begin{align*}
	 	\limsup_{N\rightarrow\infty}V_{N,k}(\bm_k)&\leq \sup_{\bb _k\in\mathcal{B} _k}\Bigl(F_{k}(\bb _k)-\frac{1}{2}\sum_{p=1}^k\beta _p^2m_p\Bigr).
	 	\end{align*}
	 	To prove the reverse direction, note that for any $\bb_k\in\mathcal{B} _k$,
	 	\begin{align}
	 	\begin{split}\label{lem12:proof:eq1}
	 	\beta _pH_{N,p}(\sigma)&=\beta _p(Nm_p)^{1/2}\cdot\frac{H_{N,p}(\sigma)}{(Nm_p)^{1/2}}\leq  \frac{\beta _p^2Nm_p}{2}+\frac{H_{N,p}(\sigma)^2}{2Nm_p},
	 	\end{split}
	 	\end{align}
	 	which implies
	 	\begin{align*}
	 	V_{N,k}(\bm_k)&\geq F_{N,k}(\bb _k)-\frac{1}{2}\sum_{p=1}^k\beta _p^2m_p,\,\,\forall \bb _k\in \mathcal{B} _k
	 	\end{align*}
	 	and thus,
	 	\begin{align*}
	 	\liminf_{N\rightarrow\infty}V_{N,k}(\bm_k)&\geq \sup_{\bb _k\in \mathcal{B} _k}\Bigl(F_k(\bb _k)-\frac{1}{2}\sum_{p=1}^k\beta _p^2m_p\Bigr).
	 	\end{align*}
	 	This finishes our proof.
	 \end{proof}
	 
	 As we have seen from Lemma \ref{lem12}, the variational representation of $V$ is based on the assumptions (A0), (A1) and (A2). In what follows, we continue to provide the proof of Theorem \ref{duality}. As we will see below, the same conditions in Lemma \ref{lem12} also yield \eqref{duality:eq1}. However, the validity of the formula \eqref{duality:eq2} for $F$ will require concavity of the limiting free energy.
	 
	 \begin{proof}[\bf Proof of Theorem \ref{duality}] Let $\bm\in\mathcal{M} $ and let $\bm_k$ be the projection of the first $k$-coordinates of $\bm$. Note that from Lemmas \ref{add2} and \ref{lem12},
	 	\begin{align*}
	 	V_k(\bm_k)&=\sup_{\bb_k\in \mathcal{B} _k}\Bigl(F_k(\bb_k)-\frac{1}{2}\sum_{p=1}^k \beta_p^2 m_p\Bigr)\leq \sup_{\bb\in \mathcal{B} }\Bigl(F(\bb)-\frac{1}{2}\sum_{p\geq 1} \beta_p^2 m_p\Bigr).
	 	\end{align*} 
	 	Since $V_{N,k}(\bm_k)\leq V_{N,k+1}(\bm_{k+1})$ for all $k\geq 1,$ we may use the monotone convergence theorem to get
	 	\begin{align*}
	 	V_N(\bm)&\leq \sup_{\bb\in \mathcal{B} }\Bigl(F(\bb)-\frac{1}{2}\sum_{p\geq 1} \beta_p^2 m_p\Bigr)
	 	\end{align*}
	 	and passing to limit,
	 	\begin{align*}
	 	\limsup_{N\rightarrow\infty}V_N(\bm)&\leq \sup_{\bb\in \mathcal{B} }\Bigl(F(\bb)-\frac{1}{2}\sum_{p\geq 1} \beta_p^2 m_p\Bigr).
	 	\end{align*}
	 	On the other hand, from \eqref{lem12:proof:eq1}, we could also obtain that for any $\bb\in\mathcal{B} $,
	 	\begin{align*}
	 	V_N(\bm)&\geq F_N(\bb)-\frac{1}{2}\sum_{p\geq 1} \beta_p^2m_p,
	 	\end{align*}
	 	from which letting $N\rightarrow\infty$ and then taking supremum,
	 	\begin{align*}
	 	\liminf_{N\rightarrow\infty}V_N(\bm)&\geq \sup_{\bb\in \mathcal{B} }\Bigl(F(\bb)-\frac{1}{2}\sum_{p\geq 1} \beta_p^2 m_p\Bigr).
	 	\end{align*}
	 	This finishes the proof of $(i).$ 
	 	
	 	As for $(ii),$ we assume that $F$ is G\^{a}teaux differentiable on $\mathcal{B} $ and $F(\bb^{1/2})$ is concave on $\mathcal{B} .$ Let $\bb\in \mathcal{B} $ and set $C(\bb)=\{p\geq 1:\partial_{p}F(\bb)=0\}.$ Take any $\bm'\in\mathcal{M} $ with $m_p'=\beta_p^{-1}\partial_pF(\bb)$ for all $p\notin C(\bb).$ From $(i),$ we rewrite
	 	\begin{align}
	 	\begin{split}\label{duality:proof:eq1}
	 	V(\bm')&=  \sup_{\bg\in\mathcal{B} }L(\bg^2)
	 	\end{split}
	 	\end{align}
	 	for 
	 	$$
	 	L(\boldsymbol{t}):=F(\boldsymbol{t}^{1/2})-\frac{1}{2}\sum_{p\geq 1} t_p m_p'.
	 	$$
	 	Since $F(\bb^{1/2})$ is concave in $\bb\in\mathcal{B} $, the function $L$ is concave as well. As now,
	 	\begin{align*}
	 	\partial_pL(\boldsymbol{t})&=\frac{1}{2}\Bigl(\frac{\partial_pF(\boldsymbol{t}^{1/2})}{t_p^{1/2}}-m_p'\Bigr),
	 	\end{align*}
	 	if we let $\boldsymbol{t}={\bb}^2$, from our choice of $\bm',$ this partial derivative is equal to $0$ for all $p\notin C(\bb)$ and moreover, it is equal to $-m_p'/2$ for all $p\in C(\bb).$ It then follows from the concavity of $L$ that
	 	\begin{align*}
	 	L(\bg^2)-L({\bb}^2)&\leq \sum_{p\geq 1}\partial_pL({\bb}^2)(\gamma_p^2-\beta_p^2)\\ 
	 	&=-\frac{1}{2}\sum_{p\in C(\bb)}m_p'(\gamma_p^2-\beta_p^2)\\
	 	&\leq \frac{1}{2}\sum_{p\in C(\bb)}m_p'\beta_p^2,\qquad\forall \bg\in\mathcal{B} ,
	 	\end{align*}
	 	from which
	 	\begin{align}
	 	\label{duality:proof:eq2}
	 	V(\bm')\leq F(\bb)-\frac{1}{2}\sum_{p\geq 1} \beta_p^2 m_p'+\frac{1}{2}\sum_{p\in C(\bb)}m_p'\beta_p^2
	 	\end{align}
	 	and thus,
	 	\begin{align*}
	 	\inf_{\bm\in\mathcal{M} }\Bigl(V(\bm)+\frac{1}{2}\sum_{p\geq 1} \beta_p^2 m_p\Bigr)&\leq V(\bm')+\frac{1}{2}\sum_{p\geq 1} \beta_p^2 m_p'\\
	 	&\leq F(\bb)+ \frac{1}{2}\sum_{p\in C(\bb)} \beta_p^2 m_p'.
	 	\end{align*}
	 	Since this inequality is valid for any $\bm'\in\mathcal{M} $ with $m_p'=\beta_p^{-1}\partial_pF(\bb)$ for all $p\notin C(\bb).$ Letting $m_p'\downarrow 0$ for all $p\in C(\bb)$, we obtain
	 	
	 	\begin{align*}
	 	\inf_{\bm\in\mathcal{M} }\Bigl(V(\bm)+\frac{1}{2}\sum_{p\geq 1} \beta_p^2 m_p\Bigr)&\leq F(\bb).
	 	\end{align*}
	 	Finally, from \eqref{lem12:proof:eq1}, it can be easily seen that
	 	\begin{align*}
	 	F(\bb)&\leq \inf_{\bm\in\mathcal{M} }\Bigl(V(\bm)+\frac{1}{2}\sum_{p\geq 1} \beta_p^2 m_p\Bigr).
	 	\end{align*}
	 	Our proof is completed by this and last inequalities.
	 \end{proof}

	 \begin{proof}[\bf Proof of Theorem \ref{thm:op}]
	 	In view of the second-half of the proof of Theorem \ref{duality} and the notation $C(\bb)$ therein, if  $\bb\in\mathcal{B} _0$ and $\bm\in\mathcal{M} _0$ satisfy $m_p=\beta_p^{-1}\partial_pF(\bb)$ for all $p\geq 1,$ then $C(\bb)=\emptyset$ and from \eqref{duality:eq1} and \eqref{duality:proof:eq2},
	 	$$
	 	V(\bm)=F(\bb)-\frac{1}{2}\sum_{p\geq 1} \beta_p^2m_p.
	 	$$ 
	 	This clearly gives $(i)$ and $(ii)$ by using \eqref{duality:proof:eq1} and \eqref{duality:proof:eq2}.
	 \end{proof}

\end{document}